\documentclass{amsart}
\usepackage{graphicx}
\usepackage{enumitem}
\usepackage{ytableau}
\usepackage{amssymb}
\usepackage{stackrel}
\usepackage{url}

\newtheorem{theorem}{Theorem}[section]
\newtheorem{lemma}[theorem]{Lemma}

\theoremstyle{definition}
\newtheorem{definition}[theorem]{Definition}

\theoremstyle{remark}

\numberwithin{equation}{section}

\newcommand{\CC}{\mathbb{C}}

\newcommand{\ZZ}{\mathbb{Z}}
\newcommand{\End}{\text{End}}
\newcommand{\Aut}{\text{Aut}}

\begin{document}

\title[Normal Forms of endomorphism-valued power series]{Normal Forms of endomorphism-valued power series}

\author{Christopher Keane and Szil\'ard Szab\'o}
\address{Department of Mathematics, Reed College, Portland OR, 97202 and Budapest University of Technology and Economics, Hungary}
\email{chkeane@reed.edu, szabosz@math.bme.hu}
\thanks{ This paper is the product of the Research Opportunities course at the Budapest Semesters in Mathematics program.}

\date{\today}

\begin{abstract}
We show for $n,k\geq1$, and an $n$-dimensional complex vector space $V$ that if an element $A\in\text{End}(V)[[z]]$ has constant term similar to a Jordan block, then there exists a polynomial gauge transformation $g$ such that the first $k$ coefficients of $gAg^{-1}$ have a controlled normal form. Furthermore, we show that this normal form is unique by demonstrating explicit relationships between the first $nk$ coefficients of the Puiseux series expansion of the eigenvalues of $A$ and the entries of the first $k$ coefficients of $gAg^{-1}$.
\end{abstract}

\maketitle
\section*{Introduction}
From Galois theory, we know that polynomials of degree greater than 4 are not solvable by radicals. 
So finding the eigenvalues of a companion matrix of the form
\[
\begin{pmatrix}
0 & 1 & 0 & \cdots & 0\\
0 & 0 & 1 & \cdots & 0\\
0 & 0 & 0 & \ddots & \vdots\\
0 & 0 & 0 & \cdots & 1\\
\beta_{n-1} & \beta_{n-2} & \beta_{n-3} & \cdots & \beta_0
\end{pmatrix}
\]
algebraically in terms of the $\beta_i$'s is not possible. 
If, however, the $\beta_i$'s have expansions $\beta_i(z)$ in terms of some other variable $z$ with $\beta_i(0)=0$, we may then ask to find the coefficients in the series expansions of these eigenvalues in terms of these $\beta_i(z)$'s. 

In this paper we work with a formal power series $A\in\text{End}(V)[[z]]$ whose constant term is a regular nilpotent endomorphism. 
We want to compute the coefficients of the Puiseux expansion of the eigenvalues of $A$, but since this is not possible algebraically 
we search for some normal form obtained via conjugating by an invertible transformation. 
Clearly, conjugating does not modify the eigenvalues of $A$, 
and our aim is to conjugate $A(z)$ to a simple 
shape that allows us to compute explicit relationships between coefficients of the series expansion of the eigenvalues and the coefficients of the conjugate. 

In \cite{szabo} this problem arose in taking an endomorphism of a vector bundle with some fixed local behavior and searching for 
the base locus of its corresponding spectral curves. 
They work with the special case of rank 2 vector bundles $E$ and irregular Higgs fields $\theta(z)$, 
i.e. meromorphic sections of the endomorphism bundle of $E$ tensored by the canonical bundle. 
Specifically, the endomorphism $\theta$ is assumed to have a single pole of order 4 at $z=0$  with leading-order term having non-trivial 
nilpotent part, and the authors show that its polar part may be brought to a simple form up to applying some holomorphic gauge transformations. 
The authors also note that the case of endomorphisms having two distinct eigenvalues is much simpler. 
Let us point out that the rank 2 cases can be tackled algebraically due to the existence of the quadratic formula, 
but that method breaks down in higher-rank cases for the Galois-theoretic reason alluded to above. 
Another observation is that up to a shift of the index of summation, it is equivalent to consider power series or 
Laurent series with a fixed finite pole order. Therefore, in this paper we content ourselves with working with power series, 
however the role of the pole order (the number of terms in the normal form to be controlled) is played by our parameter $k$. 

Here we cover the general rank $n$ case for endomorphism-valued power series with leading-order term a regular 
nilpotent endomorphism. 
That is, we maintain the assumptions of \cite{szabo}, aside from the pole of order 4 and the rank being equal to 2, 
extending their results to vector bundles of arbitrary rank and an arbitrary number of terms in the expansion of the 
endomorphism, by presenting existence and uniqueness 
statements for the normal form of endomorphism-valued power series. This has the same consequence as in \cite{szabo} 
concerning the base locus of generic irregular Higgs bundles with regular nilpotent leading-order term. 

This question is significantly more involved if the constant coefficient of $A$ is a regular matrix with more than one eigenvalue, and even more so if the constant coefficient of $A$ is not regular. The next step we would take to obtain future results would be to examine the case of the constant term of $A$ being regular with more than one eigenvalue.
\section{Preliminaries: Endomorphisms, Gauge Transformations, Puiseux Series}
In this section we describe what kinds of endomorphisms and gauge transformations we plan to examine.
\subsection{Constraints on Endomorphisms}
We begin by putting constraints on the endomorphisms we want to examine. We remark that the results in this paper hold over any algebraically closed field of characteristic zero, but we will only be considering vector spaces over $\CC$. Let $V$ be a vector space over $\CC$ of dimension $n$. Suppose that $z$ is a complex variable, and let $A\in\End(V[[z]])$, that is, $A$ has the form
\[
    A(z)=\sum_{m=0}^\infty A_mz^m,\text{ with $A_m\in M_{n,n}(\CC)$}.
\]
We observe $A_0=A(0)$. We also place the following condition of regularity on $A_0$.

\begin{definition}
For a vector space $V$ over an algebraically closed field, an $n\times n$ matrix $A_0$ is \textit{regular} if and only if its Jordan normal form is of the form
\[
J_{d_1}(\lambda_1)\oplus\cdots\oplus J_{d_s}(\lambda_s),
\]
with $i\neq j\implies\lambda_i\neq\lambda_j$, and where each $J_{d_i}(\lambda_i)$ is a Jordan block of size $d_i$ with corresponding eigenvalue $\lambda_i$.
\end{definition}

More abstractly, this is equivalent to considering the space of complex $n\times n$ matrices as a Lie algebra and requiring that the centralizer of $A_0$ has minimal dimension. The importance of this will become clearer later with the discussion of the transformation applied to $A$.
\subsection{Constraints on Gauge Transformations}\label{gtransf}
Consider $g\in\Aut(V)[[z]]$, supposing that $g$ has a power series expansion
\begin{equation*}
g(z)=\sum_{m=0}^\infty g_mz^m,\text{ with $g_m\in M_{n,n}(\CC),g_0\in \text{GL}_n(\CC)$}.
\label{form of B}
\end{equation*}
We call $g$ an ``analytic/formal gauge transformation" (according to whether the radius of convergence of the power-series is $0$ or positive), 
and require that $g_0$ be invertible because we intend to conjugate $A$ by $g$. 
It is a well-known fact about rings of formal power series that an element is invertible if and only if its constant term is invertible. 
Since $g$ is a power series of matrices, this means we must have $g_0\in\text{GL}_n(\CC)$ for $g$ to be invertible.

We turn our attention to the conjugation of $A$ by $g$, and rename it $B$:
\begin{equation}
    g(z)A(z)g^{-1}(z)=B(z)=\sum_{m=0}^\infty B_mz^m.
\end{equation} 
Our first goal is to design $g$ such that we may control any finite number of the matrix coefficients in the conjugation. Because eigenvalues are invariant under conjugation, transforming $A$ into $B$ will make computation of the eigenvalues of $A$ simpler. We obtain the following theorem, which will be restated later as Theorem \ref{normform}:
\begin{theorem}\label{prenormform}
Suppose that $k,n\geq1$, $V$ is an $n$-dimensional vector space over $\CC$, and $A\in\text{End}(V)[[z]]$. 
If $A$ is such that $A_0$ is similar to a Jordan block with eigenvalue 0, then we may construct a polynomial gauge transformation $g$ 
such that $B_0$ is an upper triangular Jordan block of dimension $n$ and the first $k$ coefficients $B_1,\ldots , B_k$ of $gAg^{-1}=B$ 
are matrices with nonzero coefficients only in their $n$'th row.
\end{theorem}
The series $B$ will be referred to as ``the normal form" from now on. 
With the existence of this established we move towards our second goal of determining explicit relationships between 
the eigenvalues of $A$ and the entries of the coefficients of $B$. Let us enumerate the possibly non-zero entries of $B_m$ 
from left to right as $b_{mn-n+1}, \ldots , b_{mn}$. 
We obtain the following result, which will be restated later as Theorem \ref{eigvals}:
\begin{theorem}
Let $B$ be the normal form of $A$ as described in Theorem \ref{prenormform}, and suppose that the bottom left coefficient $b_1$ of 
$B_1$ determined by the normal form is nonzero. The eigenvalues of $A$ have a Puiseux expansion
\[
\zeta(z)=\sum_{m=1}^\infty a_mz^{m/n},
\]
and for fixed $s\geq1,$ the first $s$ coefficients $a_1, \ldots ,a_s$ of the Puiseux expansion explicitly determine and are determined 
by the first $s$ entries $b_1, \ldots ,b_s$ of the matrices making up the normal form $B$. 
\end{theorem}
In particular, 
this theorem tells us that for fixed $k$ the normal form $B$ of $A$ is uniquely determined. 
In all cases we assume $A(z)=\sum_{m=0}^\infty A_mz^m$ is such that $A_0$ is similar to a Jordan block. 
Thus we may define $g_0\in\text{GL}_n(\CC)$ such that 
$$
  B_0 = g_0 A_0 g_0^{-1}
$$
has the desired Jordan block form. 
This is a constant transformation, which is notable since the final $g$ will be a finite product of polynomials. 
Specifically, we will build $g$ as a product of $g_0$ introduced above and non-constant factors $h_\ell$ of the form 
\begin{equation*}
h_\ell(z)=I_n+g_\ell z^\ell,
\end{equation*}
where $I_n$ is the $n\times n$ identity matrix and $1\leq \ell\leq k\in\ZZ^+$. 
This is an important point, because it means that $g$ will be a polynomial, hence everywhere convergent, so applying them to $A$ 
will not affect the convergence radius of $A$. 
This means that the portion of our results concerning gauge transformations will apply to rings of power series where convergence 
is a relevant concern. 
Furthermore, since we only consider the terms of $A$ up to the $k$'th degree we will be applying $k$ of these $h_\ell$ transformations, 
so instead of computing an explicit form for $g^{-1}$, we will only need that $h_\ell^{-1}(z)=I_n-g_\ell z^\ell+O(z^{\ell+1})$. 
Then conjugation of $A$ by one of the factors $h_\ell$ looks like:
\begin{align*}
    h_\ell(z)A(z)h_\ell^{-1}(z)&=(I_n+g_\ell z^\ell)\left(\sum_{m=0}^\infty A_mz^m\right)(I_n-g_\ell z^\ell)+O(z^{\ell+1})\\
    &=\left(\sum_{m=0}^{\ell-1}A_mz^m\right)+(A_\ell-[A_0,g_\ell])z^\ell+O(z^{\ell+1}),
\end{align*}
where $[A_0,g_\ell]=A_0g_\ell-g_\ell A_0$ represents the commutator. 
In this manipulation we see that $g$ affects the $\ell$'th term of $A$ without changing the first $\ell-1$ terms. 
This is important because we apply the transformations $I_n-g_\ell z^\ell$ iteratively for $1\leq \ell\leq k$ for $\ell$ \textit{increasing}, ultimately obtaining a polynomial transformation of the form
\begin{equation}\label{gauge}
g(z)=h_k(z)h_{k-1}(z)\cdots h_1(z)g_0=(I_n+g_kz^k)(I_n+g_{k-1}z^{k-1})\cdots(I_n+g_1z)g_0.
\end{equation}
Specifically, considering the map
\begin{align}
    \text{ad}_{A_0}:M_{n,n}(\CC)&\to M_{n,n}(\CC)\nonumber\\
    g_\ell&\mapsto[A_0,g_\ell]=A_0g_\ell-g_\ell A_0\label{commutator},
\end{align}
will tell us about how to construct $g$ to generate a normal form for the conjugated series.
\subsection{Factorization of the Characteristic Polynomial of $A$}

We consider the eigenvalues of endomorphisms in the variable $\zeta$. 
Let $A(z) =  \sum_{m = 0}^\infty A_{m}z^{m}$ be an element of $\text{End}(V)[[z]]$. 
We have that the characteristic polynomial of $A(z)$ has the following form: 
\begin{equation}\label{charpoly}
\chi_{A(z)}(\zeta) = \chi_{A}(z, \zeta) = \det(\zeta I - A(z)) = \zeta^n + a_{1}(z)\zeta^{n-1} + \cdots + a_{n}(z),
\end{equation}
with $a_1,\ldots,a_n\in \CC [[z]]$. We then recall the following particular case of a result attributed to Puiseux and Newton:
\begin{theorem}[Newton-Puiseux]\label{thm:NP}
The characteristic polynomial \eqref{charpoly} factors as follows:
\[
\chi_{A}(w^n,\zeta)=\prod_{i=1}^n(\zeta-\zeta_i(w)),\hspace{0.5cm}\text{with $\zeta_i\in\CC [[w]]$}.
\]
\end{theorem}
This version of the theorem is taken from Lecture 12 of \cite{abhyankar}, except for identifying the ramification index as $n$ 
instead of some unspecified divisor of $n!$; this latter identification in turn follows from \cite{Serre}, Chapter I Proposition 17. 
Indeed, according to the assumption $b_1\neq 0$ the $z$-adic valuation of $a_{n}$ is $1$, on the other hand 
the coefficients $a_1(0),\ldots ,a_{n-1}(0)$ clearly vanish as $A_0$ is a nilpotent endomorphism. 
These conditions mean that $\chi_{A(z)}$ is an Eisenstein-polynomial in $\zeta$, thus it is totally ramified, 
i.e. of ramification index $n$. 

For us the above theorem means that we may decompose the characteristic polynomial 
of $A$ into linear factors, with the roots being represented by Puiseux series. 
Furthermore, we will be able to obtain each root of the polynomial by considering all of the conjugates, in the Galois theory sense, 
of a single root, that is, by multiplying $w = z^{1/n}$ by some power of a primitive $n$'th root of unity $\omega$. 
Specifically, after a branch cut we may fix a choice $z^{1/n}$ of $n$'th root of $z$, 
and then all the roots of the characteristic polynomial are expressible in the form
\begin{equation}\label{eq:roots}
 \zeta_i(z)=\sum_{m=1}^\infty a_m (\omega^i z^{1/n})^m
\end{equation}
for $i=0,\ldots,n-1$. Different choices of $z^{1/n}$ only amount to a permutation of the $n$ roots $\zeta_i$.
\section{Existence of the normal form}
In this section we present the construction of a normal form for $A$ where the dimension of the ambient vector space $V$ is an 
arbitrary integer $n \geq 2$. Furthermore, we fix an arbitrary $k\in\ZZ^+$. 
\begin{theorem}\label{normform}
Take $V$ to be a vector space over $\CC$ of dimension $n$, and suppose that $A(z)=\sum_{m=0}^\infty A_m z^m$ is an endomorphism of $V$ such that 
$A_0$ is similar to a Jordan matrix with a single eigenvalue. Then for fixed $k\geq 1$ we may construct a gauge transformation $g$ of the form \eqref{gauge} such that the coefficient $B_0$ of $gAg^{-1}(z)=B(z)=\sum_{m=0}^\infty B_m z^m$ has the form
\[
B_0=
\begin{pmatrix}
        0 & 1 & 0 & 0 & \cdots & 0\\
        0 & 0 & 1 & 0 & \cdots & 0\\
        0 & 0 & 0 & 1 & \cdots & 0\\
        \vdots & \vdots & \vdots & \vdots & \ddots &\vdots\\
        0 & 0 & 0 & 0 & \cdots & 1\\
        0 & 0 & 0 & 0 & \cdots & 0
\end{pmatrix},
\]and the subsequent coefficients have the form
\[
B_\ell=
\begin{pmatrix}
        0 & 0 & \cdots & 0 & 0\\
        0 & 0 & \cdots & 0 & 0\\
        \vdots & \vdots & \ddots & \vdots & \vdots\\
        0 & 0 & \cdots & 0 & 0\\
        b_{n(\ell-1)+1} & b_{n(\ell-1)+2} & \cdots & b_{n\ell-1} & b_{n\ell}\\
\end{pmatrix},
\]
for $1\leq\ell\leq k$.
\end{theorem}
\begin{proof}
 We want to find a way to conjugate $A$ into $B$ such that $A_0=B_0$ and the subsequent $B_\ell$'s for $1\leq\ell\leq k$ have the indicated form. So we consider the map $\text{ad}_{A_0}:V\to V$ for an arbitrary matrix $G$ given by $G\mapsto[A_0,G]$, with the bracket representing the commutator of $A_0$ and $G$. To examine the image of this map, label the entries of $G$ in the usual way and expand:
\begin{center}\[
    \left[
        \begin{pmatrix}
            0 & 1 & 0 & 0 & \cdots & 0\\
            0 & 0 & 1 & 0 & \cdots & 0\\
            0 & 0 & 0 & 1 & \cdots & 0\\
            \vdots & \vdots & \vdots & \vdots & \ddots &\vdots\\
            0 & 0 & 0 & 0 & \cdots & 1\\
            0 & 0 & 0 & 0 & \cdots & 0
        \end{pmatrix},
        \begin{pmatrix}
            g_{11} & \cdots & g_{1n}\\
            \vdots & \ddots & \vdots\\
            g_{n1} & \cdots & g_{nn}\\
        \end{pmatrix}
    \right]
    =
\]\end{center}
\[
    \begin{pmatrix}
        g_{21} & g_{22}-g_{11} & g_{23}-g_{12} & \cdots & g_{2,n}-g_{1,n-1}\\
        g_{31} & g_{32}-g_{21} & g_{33}-g_{22} & \cdots & g_{3,n}-g_{2,n-1}\\
        \vdots & \vdots & \vdots & \ddots & \vdots\\
        g_{n,1} & g_{n,2}-g_{n-1,1} & g_{n,3}-g_{n-1,2} & \cdots & g_{n,n}-g_{n-1,n-1}\\
        0 & -g_{n,1} & -g_{n,2} & \cdots & -g_{n,n-1}
    \end{pmatrix}.
\]
Name the above matrix $C$, and name the entries in the usual way. Then see that we may write each entry in the last row as follows:
\[
c_{n,t}=-\sum_{j=1}^{t-1}c_{n-j,t-j},
\]
as $t$ ranges from 1 to $n$. That is, each entry in the last row is the negative of the sum of entries along the diagonal up and to the left of $c_{n,t}$. We set $c_{n,1}=0$ by convention. Now although we considered the matrix $G$ to be arbitrary, we may pick the entries of $G$ so that we can make $A_\ell-[A_0,G]$ have a desired form. Specifically, the dependence of the last row of $C$ on the first $n-1$ rows ensures that we can eliminate the first $n-1$ rows of $A_\ell$. This almost certainly affects the last row of $A_\ell$, but this does not matter to us. Thus from the iterative process described at the end of section \ref{gtransf}, we may find a polynomial of the form \eqref{gauge} that we may conjugate $A$ by to turn the $\ell$th coefficient of $B(z)$ into the form
\[
B_\ell=
\begin{pmatrix}
        0 & 0 & \cdots & 0 & 0\\
        0 & 0 & \cdots & 0 & 0\\
        \vdots & \vdots & \ddots & \vdots & \vdots\\
        0 & 0 & \cdots & 0 & 0\\
        b_{n(\ell-1)+1} & b_{n(\ell-1)+2} & \cdots & b_{n\ell-1} & b_{n\ell}\\
\end{pmatrix},
\]
for $1\leq\ell\leq k$. Turning $A_0$ into $B_0$ is much easier, since it is achieved by a constant transformation, and we are assuming that $A_0$ is similar to a matrix of the form $B_0$. 
This is the desired normal form for the first $k$ coefficients of $B$.
\end{proof}
\section{Uniqueness of the normal form}
In this section we again fix an arbitrary $k\in\ZZ^+$ and show that the coefficients $b_i$ for $1\leq i \leq kn$ 
are uniquely determined by the shape of the normal form $B$ and the coefficients $a_1,\ldots ,a_{kn}$ of the Puiseux expansion 
of the eigenvalues of $A$. 
We begin the search for relationships between the series of eigenvalues and the entries of the $B_\ell$'s with a lemma. 
For the remainder of this section we now suppose that $A_0=B_0$ is as in Theorem \ref{normform} and that the first $k$ 
coefficients of $B(z)$ may have nonzero entries only in the $n$'th row.

\begin{lemma}\label{lemma}
Let $t$ be an integer with $n>t\geq1$, and $w_1,\ldots,w_{t + 1} \in\ZZ$ be such that 
\[-n<w_1<0,\hspace{0.5cm} 
0<w_2, \cdots, w_{t+1}<n,\hspace{0.5cm}\sum_{\ell=1}^{t+1}w_\ell=0.\]

Define $\omega$ to be a primitive $n$'th root of unity. Then we have that
\[
\frac{1}{t!}\sum_{\substack{1\leq s_1,\ldots,s_{t+1}\leq n\\s_j\neq s_\ell,\forall \ell\neq j}}\omega^{w_1s_1+\cdots+w_{t+1}s_{t+1}}=(-1)^{t}n.
\]
\end{lemma}
\begin{proof}
First, note the following basic identity regarding sums of powers of primitive $n$'th roots of unity:
for any $w \in \mathbb{Z}$ such that $n \nmid w$ we have 
\begin{equation}\label{id}
    \sum_{j=0}^{n-1} \omega^{jw} = \frac{\omega^{wn} - 1}{\omega^w -1} = 0.
\end{equation}
For our application below, let us point out that in the sum of the left-hand side the summation index $j$ may equally be chosen to 
range from $1$ to $n$ without changing the value of the sum, because $\omega^{0w} = \omega^{nw}$. 
Then we proceed by induction on $t$. Starting with $t=1$, we see that we must have $w_2=-w_1$, since $w_1<0$, and $w_1+w_2=0$. Then see that
\[
\frac{1}{1!}\sum_{\substack{1\leq s_1,s_2\leq n\\s_1\neq s_2}}\omega^{w_1(s_1-s_2)},
\]
and relabelling $u=(s_1-s_2)\bmod n$ gives
\[
n \cdot \frac{1}{1!}\sum_{u=1}^{n-1}\omega^{w_1u}=n\cdot\frac{1}{1!}\cdot(-1)=(-1)^1 \cdot n,
\]
using (3.1) and observing each $u$ is obtained in $n$ possible ways. So the base case is proven.

Now suppose that the claim holds for $t-1\geq1$. For $t$, we then have
\[
\frac{1}{t!}\sum_{\substack{1\leq s_1,\ldots,s_{t+1}\leq n\\s_j\neq s_\ell,\forall \ell\neq j}}\omega^{w_1s_1+\cdots+w_{t+1}s_{t+1}}
=
\frac{1}{t!}\sum_{\substack{1\leq s_2,\ldots,s_{t+1}\leq n\\s_j\neq s_\ell,\forall \ell\neq j}}\omega^{w_2s_2+\cdots+w_{t+1}s_{t+1}}\left(\sum_{\substack{s_1=1\\s_1\notin\{s_2,\ldots,s_{t+1}\}}}^n\omega^{w_1s_1}\right)
\]
and since $w_1\not\equiv 0\bmod n$, we may re-write the inner sum using \eqref{id}:
\begin{align*}
    &\frac{1}{t!}\sum_{\substack{1\leq s_2,\ldots,s_{t+1}\leq n\\s_j\neq s_\ell,\forall \ell\neq j}}\omega^{w_2s_2+\cdots+w_{t+1}s_{t+1}}\left(\sum_{\substack{s_1=1\\s_1\notin\{s_2,\ldots,s_{t+1}\}}}^n\omega^{w_1s_1}\right)
    =\\
    &=
    \frac{1}{t!}\sum_{\substack{1\leq s_2,\ldots,s_{t+1}\leq n\\s_j\neq s_\ell,\forall \ell\neq j}}\omega^{w_2s_2+\cdots+w_{t+1}s_{t+1}}\left(-\omega^{w_1s_2}-\cdots-\omega^{w_1s_{t+1}}\right)\\
    &= - \frac{1}{t!}\sum_{\substack{1\leq s_1,\ldots,s_{t+1}\leq n\\s_j\neq s_\ell,\forall \ell\neq j}}\omega^{(w_2+w_1)s_2+w_3s_3+\cdots+w_{t+1}s_{t+1}}-\cdots\\
    &-\frac{1}{t!}\sum_{\substack{1\leq s_1,\ldots,s_{t+1}\leq n\\s_j\neq s_\ell,\forall \ell\neq j}}\omega^{w_2s_2+(w_1+w_3)s_3+w_4s_4\cdots+w_{t+1}s_{t+1}}-\cdots\\
    &-\frac{1}{t!}\sum_{\substack{1\leq s_1,\ldots,s_{t+1}\leq n\\s_j\neq s_\ell,\forall \ell\neq j}}\omega^{w_2s_2+w_3s_3+\cdots+w_t s_t+(w_1+w_{t+1})s_{t+1}}.
\end{align*}
In each of the $t$ terms in the final sum, we may relabel the indices $w_1',w_2',\ldots,w_{t+1}'$ such that $w_1'=w_1+w_\ell$ for $\ell=1,\ldots,t+1$. The remaining $w_j'$ are assigned lexicographically according to what is left, that is, if $w_1'$ takes the $\ell$th spot in the list, then \[w_2'=w_2,w_3'=w_3,\ldots,w_{\ell-1}'=w_{\ell-1},w_\ell'=w_{\ell+1},\ldots,w_{t-1}'=w_{t},w_t'=w_{t+1}.\]
These relabeled terms still satisfy $\sum_{u=1}^t w_s=0$ since the original $w$ terms satisfy this relation. They also satisfy $-n < w_1' < 0$ and $0 < w_2', \cdots, w_t' < n$. This is clear for $w_j'$ with $j>1$, and also holds for $w_1'$ since we have 
\[
w_1'=w_1+w_\ell<\sum_{j=1}^{t+1}w_j=0
.\]
So we may apply the induction assumption to each of these sums to turn the last expression in the above manipulation to
\begin{align*}
    &-\frac{1}{t}\left(\frac{1}{(t-1)!}(-1)^{t-1}(t-1)! \cdot n +\frac{1}{(t-1)!}(-1)^{t-1}(t-1)! \cdot n+\cdots+\frac{1}{(t-1)!}(-1)^{t-1}(t-1)! \cdot n \right)=\\
    &=(-1)^t \cdot n.
\end{align*}
\end{proof}
This lemma is crucial in determining the coefficients we're ultimately looking for. We now present the argument for the coefficient relationships of the rank $n$ case.

Let $k\geq1$, $A\in\text{End}(V)[[z]]$ have $A_0$ similar to a Jordan block and have normal form $B$ as in Theorem \ref{normform} 
with $b_1\neq0$. Letting
\[
\zeta(z)=\sum_{m=1}^\infty a_m z^{m/n}
\]
denote the Puiseux expansion of the eigenvalues of $A$, our aim is to show that the coefficients $\{a_1,\ldots,a_s\}$ 
determine and are determined by $\{b_1,\ldots,b_s\}$ for arbitrary $1\leq s \leq kn$. 
More precisely, writing 
\begin{equation}\label{eq:snlt}
  s=n\ell-t 
\end{equation}
for a unique $1\leq\ell\leq k, 0\leq t\leq n-1$, we have the following. 
\begin{theorem}\label{eigvals}
With the above assumptions, there exist polynomials $P_{s,n}\in \CC[x_1,\ldots,x_{s-1}]$ only depending on $s,n$ such that we have
\[
b_{s}=(-1)^t na_1^t a_{s}+P_{s,n}(a_1,\ldots,a_{s-1}). 
\]
Conversely, there exist rational functions of the form $Q_{s,n}\in \CC[x_1^{\pm 1},\ldots,x_{s-1}]$ 
such that 
\[
  a_{s}=\frac{(-1)^{s}}{n}b_1^{-s/n}b_{s}+Q_{s,n}(b_1^{1/n},\ldots,b_{s-1}).
\]
In particular, for any given $A\in\text{End}(V)[[z]]$ and fixed $k$, the parameters $\{b_1,\ldots,b_{kn}\}$ appearing in Theorem 
\ref{normform} are uniquely determined. 
\end{theorem}
\begin{proof}
Let $\omega$ be a primitive $n$th root of unity and recall our notation (\ref{eq:roots}) for the eigenvalues of $A$. 
The key idea is to compare two different representations for the characteristic polynomial $\chi_{B(z)}(\zeta) = \chi_{A(z)}(\zeta)$. 
Namely, up to order $k$ with respect to the variable $z$ the polynomial $\chi_{B(z)}$ can be read off directly from the form 
of the matrices $B_0, B_1, \ldots ,B_k$ given in Theorem \ref{normform}. 
On the other hand, as we have seen in Theorem \ref{thm:NP} we may expand $\chi_{A(z)}$ into linear factors $(\zeta-\zeta_i(z))$. 
This provides us the identity 
\begin{align}
&\zeta^n+\zeta^{n-1}\left( \sum_{\ell=1}^k b_{n\ell}z^\ell + O(z^{k+1}) \right) + \cdots + 
      \left( \sum_{\ell=1}^k b_{n\ell-(n-1)}z^\ell + O(z^{k+1}) \right)\label{LHS}\\
&=\prod_{i=0}^{n-1}\left(\zeta-\zeta_i(z)\right)\nonumber\\
&=\left( \zeta-\sum_{m=1}^\infty a_mz^{m/n} \right) \left( \zeta-\sum_{m=1}^\infty a_m(\omega z^{1/n})^{m} \right) 
  \cdots \left( \zeta-\sum_{m=1}^\infty a_m(\omega^{n-1} z^{1/n})^{m} \right)\label{RHS}.
\end{align}
The generic term of \eqref{LHS} is 
\[
 \zeta^{n-1-t} \left(\sum_{\ell=1}^k b_{n\ell-t}z^\ell + O(z^{k+1}) \right).
\]
We proceed now by comparing coefficients of \eqref{LHS} and \eqref{RHS}, and to do this we apply induction on $s$. 

Before starting the induction, we do some preliminary work in computing the coefficient in \eqref{RHS} of $\zeta^{n-1-t}z^\ell$, 
that is, the coefficient that corresponds to $b_{n\ell-t}$ in \eqref{LHS}. 
We exclude the case where $\ell=1,t=n-1$ (i.e. $b_1$), since this first nonzero term has simpler combinatorial structure than 
subsequent ones. 
We would like to have a general form for the subsequent terms. 

To this end, we know that the coefficient of $\zeta^{n-1-t}z^\ell$ in \eqref{RHS} will be a complex linear combination of the 
products $a_{m_1}\cdots a_{m_{t+1}}$ such that $\sum_{i=1}^{t+1}m_i=n\ell$, with constants given in terms of a sum of powers 
of $\omega$. This is equivalent to noticing that the indices $m_i$ partition $n\ell$ into $t+1$ nonempty parts. 
To explain why there are $t+1$ parts, we first see that $n-1-t=n-(t+1)$, and in the expansion \eqref{RHS}, 
each term will have $n$ components. 
These components are formed by picking one term from each of the $n$ factors in \eqref{RHS}, and are thus split into those 
that are just $\zeta$'s and those that come from the $a_i$'s. 
In the particular case of $\zeta^{n-1-t}$ we can imagine that we use $n-1-t$ choices on $\zeta$'s, 
and the remaining $t+1$ choices on various $a_{m_i}$'s. 
The correct coefficient in \eqref{RHS} to compare to $b_{n\ell-t}$ will be then those combinations of $a_{m_i}$'s such that 
the indices $m_i$ sum to $n$ times the exponent of $z$ multiplying $b_{n\ell-t}$, that is, the $m_i$ sum to $n\ell$. 
We see that the parts must be nonempty since any $m_i=0$ would give us a factor of $a_0=0$ in the product of all $a_{m_i}$'s, 
thus annihilating the product.

So we need to consider the set of all partitions of the integer $n\ell$ as a sum of $t+1$ positive integers say in decreasing order: 
\[ 
 \mathcal{P}_{\ell, t} = \{ m_1\geq \ldots \geq m_{t+1} \geq 1 | \quad m_1+\cdots+m_{t+1}=n\ell \}.
\]
With this notation we can produce an initial expression for the general coefficient: 
\begin{equation}
   b_{n\ell-t} = \sum_{\mathcal{P}_{\ell, t}}
   a_{m_1}\cdots a_{m_{t+1}}\mu_{m_1,\ldots,m_{t+1}},
    \label{init}
\end{equation}
where $\mu_{m_1,\ldots,m_{t+1}}$ denotes a yet undetermined linear combination of powers of $\omega$ with rational coefficients 
that depends on the partition $( m_1, \ldots , m_{t+1} )$.

The expression in \eqref{init} can be refined by noticing that we care only about the partitions with $m_1=n\ell-t = s$, 
since this will be the highest possible index for a given $s$ and given $t$, and the products coming from all partitions with 
$m_1 < n\ell-t$ will be absorbed in the polynomial $P_{s,n}(a_1^{1/n},\ldots,a_{s-1})$. 
This assignment of $m_1$ then necessarily forces $m_2=\cdots=m_{t+1}=1$, since we still require that the partition contains 
$t+1$ nonempty parts and that the $m_i$ sum to $n\ell$. 
Let us now introduce 
\[ 
   \mathcal{P}^0_{\ell, t} = \{ ( m_1, \ldots , m_{t+1} ) \in \mathcal{P}_{\ell, t} | \quad n\ell-t > m_1 \}. 
\]
This $\mathcal{P}^0_{\ell, t}$ captures all of the partitions whose $m_1$ index we do not need to keep track of, allowing us to re-write \eqref{init}. In re-writing we suppress the $m_i$ in the first term, instead presenting their actual values which we know to be 
$m_1=n\ell-t,m_2=\cdots=m_{t+1}=1$:
\begin{equation}
    a_1^t a_{n\ell-t}\mu_{n\ell-t,1,\ldots,1}+
    \sum_{\mathcal{P}^0_{\ell, t}}
      a_{m_1}\cdots a_{m_{t+1}}\mu_{m_1,\ldots,m_{t+1}}
    \label{sep}.
\end{equation}
Again, as the indices $m_i$ of each term in the sum are all strictly less than $n\ell-t$, the second term in this formula 
only contributes to $P_{s,n}$, hence we only need to specify the constants $\mu_{n\ell-t,1,\ldots,1}$. 

To gain a better understanding of the structure of the constant $\mu_{n\ell-t,1,\ldots,1}$ appearing in the above expression 
we describe a way of visualizing each partition that will give more structure to the enumeration of the constant's summands. 
Consider the partition of $n\ell$ into parts $n\ell-t,1,\ldots,1$ with $1$ appearing $t$ times. 
We align this partition with the combinatorial choice of picking a term out of each of the $n$ factors of \eqref{RHS} 
by considering the $m_i$ to be distributed among $n$ boxes, not necessarily in increasing order. 
We label the positions of these $m_i$ amongst the $n$ boxes by the labels $s_i$ for $i=1,\ldots,t+1$, such that 
$s_i\neq s_j$ for $i\neq j$. 
Observe however that since $m_2 = \cdots = m_{t+1}$, any fixed set $\{ s_2 , \ldots , s_{t+1} \}$ 
of $t$ distinct positions in $\{ 1, \ldots ,n \}$ and any further position $s_1 \notin \{ s_2 , \ldots , s_{t+1} \}$ 
give rise to a \emph{single} term in \eqref{sep} of the form $\omega^{s}a_1^t a_{n\ell - t}$ for some integer $s$ (to be specified below), 
independently of the order of $\{ s_2 , \ldots , s_{t+1} \}$. 
So we may (and from now on, will) assume that the positions $\{ s_2 , \ldots , s_{t+1} \}$ are in increasing order 
\[
 s_2 < \cdots < s_{t+1}; 
\]
however, we have no restriction about the position of $s_1$ relative to the above increasing sequence. 
This gives us a way of picturing all possible configurations of the $m_i$. 
An example of one of these configurations is as follows:
\begin{center}
\ytableausetup
{mathmode, boxsize=2em}
\begin{ytableau}
\zeta & \zeta & x_2 &\none[\dots] & \zeta & x_j & \zeta & \none[\dots] & \zeta & x_1 &\zeta & \none[\dots] & x_{t} & x_{t+1} &\zeta & \none[\dots] & \zeta & \zeta\\
\none[1] & \none[2] & \none[s_2] &\none[\dots] & \none[] & \none[s_j] & \none[] & \none[\dots] & \none[] & \none[s_1] & \none[] &\none[\dots] & \none[s_t] &\none[s_{t+1}] &\none[] &\none[\dots]& \none[n-1]&\none[n]
\end{ytableau},
\end{center} 
with $x_j=-a_{m_j}(\omega^{s_j-1}z)^{m_j}$ for all $1\leq j\leq t+1$. 
We note that the $-1$ attached to each $s_i$ in the exponents occurs since the expansion in \eqref{RHS} is indexed from 0 to $n-1$, 
but we were considering the $s_i$ as elements of $\{1,\ldots,n\}$. This is a minor adjustment. 

Computing $\mu_{m_1,\ldots,m_{t+1}}$ involves writing an expression for $\mu$ that reflects the fixing of $s_1$, 
the position of $m_1$, outside of the strict ordering of the other labels. 
We express this now, adopting the standard notation $[n]=\{1,\ldots,n\}$:
\begin{equation}
    \mu_{n\ell-t,1,\ldots,1}=\sum_{\substack{s_2,\ldots,s_{t+1}\in\ZZ^+\\1\leq s_2<\cdots<s_{t+1}\leq n}}\omega^{(s_2-1)}\cdots\omega^{(s_{t+1}-1)}
    \sum_{\substack{s_1\in[n]\setminus\\ \{s_2,\ldots,s_{t+1}\}}}\omega^{(s_1-1)(n\ell-t)}
    \label{omegasplit}.
\end{equation}
Now we manipulate \eqref{omegasplit} as follows, recognizing that since $\omega$ is an $n$th root of unity, 
we may work with any of the sums in the exponents modulo $n$:
\begin{align}
    &\sum_{\substack{s_2,\ldots,s_{t+1}\in\ZZ^+\\1\leq s_2<\cdots<s_{t+1}\leq n}}
    \sum_{\substack{s_1\in[n]\setminus\\ \{s_2,\ldots,s_{t+1}\}}}\omega^{s_2+\cdots+s_{t+1}-t+s_1\ell n-s_1t-\ell n+t}\nonumber\\
    &=\sum_{\substack{s_2,\ldots,s_{t+1}\in\ZZ^+\\1\leq s_2<\cdots<s_{t+1}\leq n}}
    \sum_{\substack{s_1\in[n]\setminus\\ \{s_2,\ldots,s_{t+1}\}}}\omega^{s_2+\cdots+s_{t+1}-s_1t}\nonumber\\
    &=\sum_{\substack{s_2,\ldots,s_{t+1}\in\ZZ^+\\1\leq s_2<\cdots<s_{t+1}\leq n}}\omega^{s_2+\cdots+s_{t+1}}
    \sum_{\substack{s_1\in[n]\setminus\\ \{s_2,\ldots,s_{t+1}\}}}\omega^{-s_1t}\label{orderedsum}.
\end{align}
We may recognize \eqref{orderedsum} as an ordered version of the sum examined by Lemma~\ref{lemma}. 
Indeed, we have bounded weights that sum to zero and an exponent sum in $t+1$ terms, namely $w_1 = -t, w_2 = \cdots = w_{t+1} = 1$. 
In Lemma~\ref{lemma} we have $t+1$ unordered terms, but here we have $t$ ordered terms and one independent term. 
Multiplying \eqref{orderedsum} by $t!$ allows us to re-write it without the ordering, and allows us to apply the lemma, 
since we obtain sums over $t+1$ unordered terms. 
But then the lemma gives that dividing by $t!$ again allows us to compute the sum, and so the sum from the lemma and the sum in 
\eqref{orderedsum} are equivalent. So we find 
\[
\sum_{\substack{s_2,\ldots,s_{t+1}\in\ZZ^+\\1\leq s_2<\cdots<s_{t+1}\leq n}}\omega^{s_2+\cdots+s_{t+1}}
    \sum_{\substack{s_1\in[n]\setminus\\ \{s_2,\ldots,s_{t+1}\}}}\omega^{-s_1t}=(-1)^t n.
\]
We conclude that the leading-index term for $b_{n\ell-t}$ is $(-1)^t na_1^t a_{n\ell-t}$.

Now we can start the induction on $s$, which will actually be a double induction, first on $\ell \in \{ 1, \ldots , k \}$ 
in increasing order then on $t \in \{ 0 , \ldots , n-1 \}$ in decreasing order, see (\ref{eq:snlt}). 
We determine $b_1$ by inspection, and apply the above argument for $b_2,\ldots,b_n$. So we have
\begin{align*}
    b_1&=a_1^n\\
    b_2&=(-1)^{n-2}na_1^{n-2}a_2\\
    &\vdots\\
    b_p&=(-1)^{n-p}na_1^{n-p}a_p\\
    &\vdots\\
    b_n&=na_n.
\end{align*}
We note that each of these $b_i$ relations matches that in the theorem statement, depending on $a_1$ and  $a_i$. These relationships are certainly invertible in terms of the $a_i$:
\begin{align*}
    a_1&=\sqrt[n]{b_1}\\
    a_2&=\frac{(-1)^{2-n}b_2}{nb_1^{1-2/n}}\\
    &\vdots\\
    a_p&=\frac{(-1)^{p-n}b_p}{nb_1^{1-p/n}}\\
    &\vdots\\
    a_n&=b_n/n.
\end{align*}
We fix an $n$'th root of $b_1$ here so that everything is uniquely determined. 
Changing the choice of the root is equivalent to multiplying $a_1$ by a primitive $n$'th root of unity, which then affects all 
subsequent coefficients $a_k$ in the same way, eventually leading to a permutation of the roots $\zeta_j(z)$ in \eqref{RHS}; 
thus, fixing an $n$'th root of $b_1$ is not a restrictive choice. 
Furthermore, we note that in one direction we have the desired polynomial relations, and in the other direction we have the desired rational relations. Thus, the statement holds for $\ell = 1$ and all $t$. 

Then supposing that the claim holds for $2,\ldots,s-1$, we consider general $s$. 
From the earlier partition argument we also know that any terms $a_i$ in the full expression for $b_{s}$ that do not contain $a_{n\ell-t}$ will have indices at most $i\leq n\ell-t-1=s-1$, so applying the induction hypothesis gives that we have
\[
b_{n\ell-t}=(-1)^t na_1^t a_{n\ell-t}+\text{P}_{s,n}(a_1,\ldots,a_{s-1}).
\]
Since we have invertible relationships for the expressions contained in \\$\text{P}_{s,n}(a_1,\ldots,a_{s-1})$. This new set of relationships will also be invertible since the only new term is $(-1)^t na_1^t a_{n\ell-t}$, which is a nonzero multiple of $a_{n\ell-t}$ since we are working over a field of characteristic zero with $a_1\neq0$. So $b_{n\ell-t}$ is determined explicitly by this expression, and vice versa. Thus we have shown that the claim holds for general $s$.
\end{proof}

\bibliography{bibliography}
\bibliographystyle{amsplain}
\end{document}